\documentclass[11pt]{article}
\title{Bulk limits of the truncated circular beta ensembles}
\date{}
\author{Yun Li, Mingchang Liu and Joseph Najnudel}

\usepackage{units}
\usepackage{amsmath, amsthm, amssymb,stackengine,bbm}
\usepackage{graphicx}
\usepackage{amsmath, enumerate}
\usepackage{bbm}
\usepackage{color}
\usepackage{hyperref}
\newtheorem{theorem}{Theorem}
\newtheorem{thmx}{Theorem}

\newtheorem{proposition}[theorem]{Proposition}
\newtheorem{corollary}[theorem]{Corollary}

\newtheorem{fact}[theorem]{Fact}
\newtheorem{assumption}[theorem]{Assumption}

\theoremstyle{definition} % For roman text in the body
\newtheorem{definition}[theorem]{Definition}
\newtheorem{remark}[theorem]{Remark}

\newcommand{\eps}{\varepsilon}

\newcommand{\R}{{\mathbb R}}
\newcommand{\D}{{\mathbb D}}

\newcommand{\C}{{\mathbb C}}

\newcommand{\ev}{{\mathbb E}}
\newcommand{\pr}{{\mathbb P}}
\newcommand{\ol}[1]{\overline{#1} \!\,} %overline

\newcommand{\diag}{\operatorname{diag}}

\newcommand{\cL}{{\mathcal L}}

\newcommand{\cM}{{\mathcal M}}
\newcommand{\cH}{{\mathcal H}}
\newcommand{\cC}{{\mathcal C}}

\newcommand{\cF}{{\mathcal F}}

\newcommand{\Sineb}{\operatorname{Sine}_{\beta}}

\newcommand{\Circ}{\mathsf{Circ}}

\newcommand{\wtl}{\widetilde}

\newcommand{\mat}[4]{\left( \begin{array}{cc}
		#1 & #2  \\
		#3 & #4  \\
	\end{array} \right)}

\newcommand{\trunc}[1]{\operatorname{Trunc}(#1)}

\newcommand{\ed}{\stackrel{d}{=}}

\newcommand{\rev}[1]{\wtl{#1}}

\begin{document}
	\maketitle

	\begin{abstract}
		We study the bulk limits of the rank-one truncation of the circular beta ensembles obtained by deleting the first row and column. We prove that the normalized Blaschke products of the truncated circular beta ensembles converge in law to a constant multiple (depending on $\beta$) of the i.i.d.~Gaussian analytic function (GAF) uniformly on compacts of the unit disk. As a consequence, we identify the bulk point process limit with the zero set of GAF, which was proved in the seminal work of Peres and Vir\'ag~\cite{PV} to form a determinantal point process with the Bergman kernel. Our method also applies to the rank-one multiplicative perturbations of the circular beta ensembles in the mesoscopic regime.
	\end{abstract}
	
	\section{Introduction}
	The size $n$ circular beta ensemble with parameter $\beta>0$ is the distribution of $n$ points $\{e^{i\theta_j}, 1\le j\le n\}$ on the unit circle with  joint probability density
	\begin{equation}\label{eq:cbe}
		\frac{1}{Z_{n,\beta}}\prod_{j<k\le n}|e^{i\theta_j}-e^{i\theta_k}|^\beta, \qquad \theta_j\in [-\pi,\pi),
	\end{equation}
	where $Z_{n,\beta}$ is an explicit normalizing constant (see e.g.~Section 2.8 of \cite{Forbook}). 
	When $\beta=2$, \eqref{eq:cbe} recovers the joint eigenvalue density of a uniformly chosen $n\times n$ unitary matrix, known as the size $n$ Circular Unitary Ensemble (CUE).
	By utilizing the algebraic structure of the joint intensities, Dyson, Gaudin, and Mehta derived the local scaling limit of CUE in the 1960s. They showed that the determinantal structure is preserved in the limit, and the limit process, known as the Sine$_2$ point process, is a translation invariant determinantal point process on $\mathbb{R}$. In \cite{CNN} Chhaibi, Najnudel and Nikeghbali studied the local scaling limit of the (normalized) characteristic polynomials of CUE, and constructed a random analytic function $\zeta_2$ (called the stochastic zeta function) so that its zero set has the same distribution as the Sine$_2$ point process. 
	
	When $\beta=1,4$, the joint density \eqref{eq:cbe} corresponds to certain random unitary matrices invariant under symmetric/self-dual transformations. For general $\beta>0$, there is no similar invariant matrix model. In \cite{KillipNenciu} Killip and Nenciu (motivated by the results of \cite{DE}) constructed a family of sparse random CMV (named after Cantero, Moral, and Vel\'azquez \cite{CMV}) matrix models $\{\Circ_{n,\beta},n\ge 1\}$ in terms of the Verblunsky coefficients so that the resulting matrix has a joint eigenvalue density given by \eqref{eq:cbe}. See Section \ref{sec:background} for an overview of the CMV construction.

	In \cite{KS} Killip and Stoiciu showed that under the same scaling as in the unitary case, the circular beta ensemble has a point process limit, called the Sine$_\beta$ point process. They showed that the Sine$_\beta$ process can be characterized via its counting function using a system of coupled stochastic differential equations. Using the random operator framework, in a series of works \cite{BVBV_op,BVBV_19,BVBV_szeta} Valk\'o and Vir\'ag showed the operator level convergence of the circular beta ensembles with an explicit rate of convergence, and proved that the normalized characteristic polynomials of the circular beta ensembles converge to the random analytic functions $\zeta_\beta(\cdot)$ with zero set given by the $\Sineb$ process. See also \cite{DE,ES,RRV,BVBV,RR} for a development of the matrix models and local scaling limits for beta ensembles on the real line.

	In this paper, we focus on the \emph{rank-one truncation} of the circular beta ensembles obtained by removing the first row and column of the matrix model $\Circ_{n,\beta}$. Random unitary matrices can be viewed as evolution operators of closed quantum systems. From this perspective, truncations of unitary matrices obtained by removing a fixed number of rows and columns, naturally arise as models for open quantum systems in the regime of weak non-Hermiticity. We refer to the surveys \cite{FyodorovSommers2003,Forrester_2023} and the references therein for a detailed discussion of rank-one perturbations of random matrix ensembles. 
	
	The main goal of this project is to understand the scaling limit of the Blaschke product of the (rank-one) truncated circular beta ensembles in the bulk regime, thereby answering a question posed in \cite{LV_truncated}. 
	Before presenting the main result, we begin with a review of the known results in the $\beta=2$ case. Throughout the paper, for a square unitary matrix $\bf{M}$ we denote by $\trunc{\bf{M}}$ the matrix obtained by removing the first row and column from $\bf{M}$. Since the truncation map corresponds to a projection, the eigenvalues of the resulting matrix $\trunc{\bf{M}}$ have modulus less than or equal to one. 
	
	Consider the case when ${\bf M}_{n+1}$ is a uniformly chosen $(n+1)\times(n+1)$ unitary matrix with respect to the Haar measure on the unitary group $U(n+1)$. In \cite{ZS} {\. Z}yczkowski and Sommers proved that the joint eigenvalue density of $\trunc{{\bf M}_{n+1}}$ is given by
	\begin{equation}\label{eq:tcue}
		\frac{1}{\pi^n}\prod_{1\le j<k\le n}|z_j-z_k|^2 , \qquad z_j\in \D=\{z\in\C:|z|<1\},
	\end{equation}
	with respect to the Lebesgue measure in the unit disk. Note that {\. Z}yczkowski and Sommers \cite{ZS} also computed the joint eigenvalue distribution for the general finite rank truncation of CUE.

	From the Vandermonde structure in \eqref{eq:tcue}, one can show that the point process distributed according to \eqref{eq:tcue} is a determinantal point process with kernel function given by
	\begin{align}\label{eq:det_tcue}
		K_{\trunc{{\bf M}_{n+1}}}(z,w)=\frac{1}{\pi}\sum_{k=0}^{n-1} (k+1) z^k \bar w^k,\qquad z,w\in\D
	\end{align}
	with respect to the Lebesgue measure on the unit disk. Taking the $n\to\infty$ limit of the kernel function \eqref{eq:det_tcue} leads to the classical Bergman kernel on the unit disk
	\begin{align}\label{eq:Bergman_kernel}
		K_{\textup{Bergman}}(z,w)=\frac{1}{\pi(1-z \bar w)^2},\qquad z,w\in\D.
	\end{align}
	The Bergman kernel is related to the study of random analytic functions as follows. 
	Let $\{\xi_n,n\ge 0\}$ be a sequence of i.i.d.~standard complex Gaussian variables. Throughout the paper, a complex Gaussian variable $\xi$ is called standard if 
	$$
	\mathbb{E}[\xi] = \mathbb{E} [\xi^2] = 0, \; \mathbb{E}[|\xi|^2] = 1.$$
	We define the random Gaussian analytic function (GAF) formally as the random series
	\begin{align}\label{eq:GAF}
		f_{\textup{GAF}}(z)=\sum_{n=0}^\infty \xi_n z^n,\quad z\in\D.
	\end{align}
	In the seminal work \cite{PV}, Peres and Vir\'ag showed that the zero set of $f_{\textup{GAF}}(\cdot)$ is a determinantal point process with kernel function given by $K_{\text{Bergman}}(z,w)$. 
	
	Putting these results together implies that the eigenvalues of the truncated CUE converge in distribution \emph{without any additional scaling} to the zeros of $f_{\textup{GAF}}(\cdot)$. We call this the \emph{bulk scaling limit} of the truncated CUE. This convergence was later strengthened in Krishnapur \cite{Krishnapur} to the convergence of the Blaschke product of the truncated CUE.
	
	\begin{thmx}[\cite{Krishnapur}]\label{thm:Krishnapur}
		Let $\Lambda_n=\{z_1,z_2,\dots,z_n\}$ be the eigenvalues of the rank-one truncation of the $(n+1)\times (n+1)$ CUE, or equivalently the eigenvalues of $\trunc{\Circ_{n+1,2}}$, and let 
		\begin{align*}
			B_n(z):=\prod_{k=1}^n \frac{z-z_k}{1-z\bar z_k}
		\end{align*}
		be the Blaschke product of the eigenvalues. Then
		\begin{equation*}
			\sqrt{n}B_n(z)\Rightarrow f_{\textup{GAF}}(z)\qquad \text{in distribution as $n\to\infty$ }
		\end{equation*}
		with respect to the uniform-on-compact topology on $\D$.
	\end{thmx}
	
	Note that \cite{Krishnapur} also provides a generalization of the work \cite{PV} by connecting the limit of the eigenvalues of the finite-rank truncation of CUE with the singular points of a matrix-valued Gaussian analytic function. 
	
	% In this project, we will investigate the scaling limit of the characteristic polynomials of the general circular beta ensemble in the rank-one truncation regime. We hope to explore higher-rank truncation in future work.

	In \cite{KK} Killip and Kozhan  studied the truncated circular beta ensemble obtained by removing the first row and column of the matrix models $\Circ_{n+1,\beta}$ of Killip-Nenciu \cite{KillipNenciu}. They constructed a sparse matrix model with spectrum distributed as that of $\trunc{\Circ_{n+1,\beta}}$, and showed that the  joint eigenvalue density of $\trunc{\Circ_{n+1,\beta}}$  with respect to the Lebesgue measure on the unit disk is given by
	\begin{equation}\label{eq:tcbe}
		\frac{\beta^{n}}{(2\pi)^{n}}\prod_{1\le j,k\le n}|1-z_j\bar z_k|^{\frac{\beta}{2}-1}\prod_{j<k\le n}|z_j-z_k|^2, \qquad z_j\in \D.
	\end{equation}
	We call the resulting distribution the size $n$ \emph{truncated circular beta ensemble}. When $\beta=2$, equation \eqref{eq:tcbe} recovers the joint eigenvalue density \eqref{eq:tcue}. However, for general $\beta>0$ the truncated circular beta ensemble does not have a nice algebraic structure.

	In \cite{LV_truncated} Li and Valk\'o derived the edge scaling limit of the normalized characteristic polynomials and the point process limit of the eigenvalues near the point 1 on the boundary of the spectrum. They characterized the limiting point process as the zero set of certain random analytic function, which is related to the stochastic zeta function $\zeta_\beta(\cdot)$ constructed in \cite{BVBV_szeta}. Their approach relies on the matrix model representation of the full and truncated circular beta ensembles \cite{KillipNenciu,KK}, and the random operator framework for circular ensembles developed by Valk\'o and Vir\'ag \cite{BVBV_op,BVBV_19,BVBV_szeta}. 
	
	In \cite{LV_truncated}, Li and Valk\'o also raised the question (Problem 1 in Section 8.2) of whether the characteristic polynomial of the truncated circular beta ensemble has a nice scaling limit in the bulk regime of the spectrum. Our main result fully answers this question and can be viewed as a beta-generalization of Theorem 5 of Krishnapur \cite{Krishnapur} (see Theorem A for the statement) in the rank-one truncation regime.

	Recall the Gaussian analytic function $f_{\textup{GAF}}(\cdot)$ defined in \eqref{eq:GAF}. 
	\begin{theorem}\label{thm:main}
		Fix $\beta>0$. Let $\Lambda_n=\{z_1,z_2,\dots,z_n\}$ be  distributed as the size $n$ truncated circular beta ensemble, and let 
		\begin{align*}
			B_n(z)\equiv B_{n,\beta}(z):=\prod_{k=1}^n \frac{z-z_k}{1-z\bar z_k}
		\end{align*}
		be the corresponding Blaschke product. Then we have
		\begin{equation}\label{eq:tcbe_char_conv}
			\sqrt{n}B_n(z)\Rightarrow \sqrt{\tfrac{2}{\beta}}f_{\textup{GAF}}(z) \qquad \text{in distribution as $n\to\infty$ }
		\end{equation}
		with respect to the uniform-on-compact topology on $\D$.
	\end{theorem}
	
	Since $\beta$ appears as a multiplicative constant in \eqref{eq:tcbe_char_conv}, we obtain the bulk scaling limit of the truncated circular beta ensemble as follows.
	\begin{corollary}\label{cor:bulk}
		Let $\Lambda_n=\{z_1,z_2,\dots,z_n\}$ be distributed as a size $n$ truncated circular beta ensemble. Let $\Lambda=\{z\in\D:f_{\textup{GAF}}(z)=0\}$, then 
		$\Lambda_n\Rightarrow \Lambda$ as $n\to\infty$, in the vague topology on locally finite point measures on $\D$.
	\end{corollary}
	Corollary \ref{cor:bulk}, together with the work of Peres and Vir\'ag \cite{PV}, implies that the limiting point process is determinantal with $K_{\textup{Bergman}}(\cdot,\cdot)$ with respect to the Lebesgue measure on $\D$, and is independent of $\beta$. More precisely, the $k$-th joint intensity of the limiting process is given by
	\begin{align*}
		\rho_k(z_1,z_2,\dots, z_k) = \pi^{-k}\det\left[\frac{1}{(1-z_i\bar z_j)^2}\right]_{1\le i,j\le k}, \qquad z_j\in\D.
	\end{align*}
	Using the determinantal structure, one can compute the asymptotics of the hole probability as follows. Let $\Lambda_n$ be distributed as the size $n$ truncated circular beta ensemble, and for $0< r< 1$ let $\mathbb{D}_r=\{z\in\C:|z|< r\}$, then we have
	\begin{align*}
		\lim_{n\to\infty}\pr(\Lambda_n\cap \D_r=\emptyset) = \pr(\Lambda\cap \D_r=\emptyset) = \prod_{k=1}^\infty (1-r^{2k}).
	\end{align*}
	As $r\uparrow 1$, the hole probability tends to $\exp(-\frac{\pi^2+o(1)}{12(1-r)})$, see Corollaries 5.1.7 and 5.1.8 in \cite{HKPV} for more details.
	
	On the other hand, the finite size truncated circular beta ensemble distributed according to \eqref{eq:tcbe} has no known determinantal structure for general $\beta\neq 2$. We prove that the linear statistics of compactly supported functions of the truncated circular beta ensemble are asymptotically independent from the so-called reversed characteristic polynomial, see Section \ref{sec:lss} and Proposition \ref{prop:asymp_ind} for more details.

	Our approach also applies to the rank-one multiplicative perturbation of the circular beta ensemble. For a square matrix ${\bf M}$, consider its multiplicative perturbation
	\begin{align}\label{def:mult_pert}
		{\bf M}^{[r]}:={\bf M}\cdot \diag(r,1,1,\dots,1), \quad r\in[-1,1].
	\end{align}
	Note that when $r=0$ the eigenvalues of ${\bf M}^{[r]}$ consist exactly of the eigenvalues of the truncated matrix $\trunc{{\bf M}}$ plus an additional eigenvalue at $0$, whereas when $r=1$ we simply have ${\bf M}^{[r]}={\bf M}$. Hence, as $r$ changes from $0$ to $1$, the perturbed model ${\bf M}^{[r]}$ can be viewed as an interpolation between the truncated and unperturbed model. The joint eigenvalue distribution of ${\bf M}^{[r]}$ in the case when ${\bf M}$ is a random Haar unitary matrix was derived in \cite{Fyodorov2001}.
	In \cite{ForIpsen}, Forrester and Ipsen studied the $n\to\infty$ limit of the rank-one multiplicative perturbation of the CUE in the regime when $r=r_n=O(n^{-1/2})$. They characterized the limiting point process as the zeros of a certain random Laurent series, and calculated the one- and two-point correlations of the limiting process. 
	
	For general $\beta\neq 2$, one may define the rank-one multiplicative perturbation of the circular beta ensemble as the spectrum of $\Circ_{n,\beta}^{[r]}$ via \eqref{def:mult_pert}. The joint eigenvalue distribution of $\Circ_{n,\beta}^{[r]}$ was derived in Killip-Kozhan \cite{KK}. Using the perturbed matrix model of Killip and Kozhan, and a similar argument as in the proof of Theorem \ref{thm:main}, we obtain the scaling limit of the Blaschke product of the rank-one multiplicative perturbation of the circular beta ensemble in the mesoscopic regime $r_n=O(n^{-1/2})$. 
	
	\begin{theorem}\label{thm:perturb}
		Fix $\beta>0$, $r\in\R$ and set $r_n=\sqrt{\frac{2}{\beta n}}r$. For all sufficiently large $n$ let $B_{n,\beta}^{[r_n]}(z):=\prod_{k=1}^n \frac{z-z_k}{1-z\bar z_k}$ be the Blaschke product of the eigenvalues of the multiplicatively perturbed model $\Circ_{n,\beta}^{[r_n]}$. Then we have
		\begin{equation}\label{eq:perturb_conv}
			\sqrt{n}B_{n,\beta}^{[r_n]}(z)\Rightarrow \sqrt{\frac{2}{\beta}}\left( r \eta +\sum_{j=1}^\infty \xi_j z^j\right), \qquad \text{in distribution as $n\to\infty$}
		\end{equation}
		with respect to the uniform-on-compact topology on $\D$. Here $\xi_j,j\ge 1$ is a sequence of i.i.d.~standard complex normal random variables, and $\eta$ is an independent, uniform random variable on the unit circle. 
	\end{theorem}
	In particular, the eigenvalues of $\Circ_{n,\beta}^{[r_n]}$ converge weakly to the zeros of the random series $r+ zf_{\textup{GAF}}(z)$. This extends the result of \cite{ForIpsen} to general $\beta>0$. 
	
	\vspace{1 ex}
	
	\noindent {\bf Outline of the paper}: In Section \ref{sec:background}, we provide a brief introduction of the theory of orthogonal polynomials on the unit circle (OPUC) and the CMV construction for finitely supported probability measures on the unit circle. Section \ref{sec:proof_main} includes the proofs of Theorems \ref{thm:main} and \ref{thm:perturb} on the scaling limit of the Blaschke products of the rank-one truncated and multiplicatively perturbed circular beta ensembles in the bulk regime. Section \ref{sec:lss} discusses the asymptotic independence of the linear statistics of the truncated circular beta ensembles and the reversed characteristic polynomials.
	\vspace{1 ex}

	\noindent {\bf Acknowledgment}: YL is grateful to Benedek Valk\'o for valuable discussions. YL and JN thank the Institut Mittag-Leffler for its  hospitality, where part of this work originated during the program \textit{Random Matrices and Scaling Limits} in Fall 2024.
	\vspace{1 ex}
	
	\noindent {\bf AI use statement}: ChatGPT was used to proofread and assist with language editing of an earlier version of this paper.
	
	\section{Background}\label{sec:background}
	This section presents a brief introduction to the OPUC theory and the CMV matrix constructions for probability measures on the unit circle, as well as a review of the matrix models for the full and truncated circular beta ensembles.
	
	\subsection{OPUC theory and CMV matrix}\label{subsec:OPUC}
	We briefly review the necessary background for the OPUC theory and the CMV algorithm. We refer the reader to \cite{OPUC1} for a detailed discussion of the subject. 
	
	Let $\mu$ be a probability measure supported on $n$ distinct points $e^{i\theta_j}, 1\le j \le n$ on $\partial \D$. For $0\le k\le n-1$, the $k$-th orthogonal polynomial is defined as the unique monic polynomial of degree $k$ that is orthogonal to $1,z,\dots,z^{k-1}$ in $L^2(d\mu)$. We extend this definition to $\Phi_n(z)=\prod_{j=1}^n (z-e^{i\theta_j})$, which can be viewed as the characteristic polynomial of the measure $\mu$. Let $\Phi_k^*(z):=z^k\overline{\Phi_k(1/\bar z)}$ be the reversed polynomial of $\Phi_k$ for $0\le k\le n$, then the Szeg\H o recursion states that (see e.g.~Section 1.5, vol.~1 of \cite{OPUC1}):
	\begin{align}\label{eq:Szego1}
		\binom{\Phi_{k+1}}{\Phi_{k+1}^*}=
		% A_k Z \binom{\Phi_{k}}{\Phi_{k}^*}, \qquad \binom{\Phi_0}{\Phi_0^*}
		\mat{1}{-\bar \alpha_k}{-\alpha_k}{1} \mat{z}{0}{0}{1} \binom{\Phi_{k}}{\Phi_{k}^*}, \qquad \binom{\Phi_0}{\Phi_0^*}
		=\binom{1}{1},\qquad 0\le k\le n-1.
	\end{align}
	Here the constants $\alpha_k$, $0\le k\le n-1$ are called the Verblunsky coefficients of $\mu$. They satisfy the property that $\alpha_k\in\D$ for $0\le k\le n-2$ and $\alpha_{n-1}\in\partial \D$. In fact, the map between probability measures supported on $n$ points on $\partial \D$ and the sequence of Verblunsky coefficients $(\alpha_0, \dots, \alpha_{n-1})\in \D^{n-1}\times \partial \D$ is a homeomorphism, see e.g.~Theorem 1.7.11 of \cite{OPUC1}. Note that a  probability measure on the unit circle that is not finitely supported has an infinite sequence of Verblunsky coefficients lying within the unit disk.
	
	Using the Verblunsky coefficients, we are able to construct the so-called CMV matrix corresponding to $\mu$ as follows.
	
	\begin{definition}\label{def:CMV}
		Let $\mu_n$ be a probability measure supported on $n$ distinct points on the unit circle and let $\alpha_k,0\le k\le n-1$ be its Verblunsky coefficients obtained via the Szeg\H o recursion (\ref{eq:Szego1}). Define
		\[
		\Xi_k = \begin{pmatrix}
			\bar\alpha_k & \rho_k\\ \rho_k& -\alpha_k
		\end{pmatrix},\qquad \rho_k=\sqrt{1-|\alpha_k|^2},\qquad 0\le k\le n-2
		\]
		and set $\Xi_{-1}=(1)$ and $\Xi_{n-1}=(\bar\alpha_{n-1})$ to be $1\times 1$ matrices. Then we define the \emph{CMV matrix} corresponding to $\mu_n$ or $\{\alpha_k,0\le k\le n-1\}$ as
		\begin{align}\label{eq:CMV}
			\cC(\alpha_0,\cdots,\alpha_{n-1}) :=\cL\cM,
		\end{align}
		where $\cL $ and $ \cM$ are $n\times n$ block-diagonal matrices 
		\begin{align*}
			\cL=\operatorname{diag}\left(\Xi_0,\Xi_2\cdots,\Xi_{2\lfloor\tfrac{n-1}{2}\rfloor}\right),\qquad \cM=\operatorname{diag}\left(\Xi_{-1},\Xi_1\cdots,\Xi_{2\lfloor\tfrac{n}{2}\rfloor-1}\right).
		\end{align*}
		Note that $\cC(\alpha_0,\dots, \alpha_{n-1})$ is unitary if and only if $|\alpha_{n-1}|=1$. To deal with the rank-one truncation and multiplicative perturbation of random unitary matrices, we also allow $\alpha_{n-1}\in\D$, in which case the resulting matrix $\cC(\alpha_0,\alpha_1,\dots,\alpha_{n-1})$ becomes sub-unitary. 
	\end{definition}

	% When $\alpha_{n-1}\in\partial\D$, the unitary CMV matrix $\cC=\cC(\alpha_0,\alpha_1\dots,\alpha_{n-1})$ has the property that the spectral measure with respect to the first coordinate is given by the finitely supported probability measure $\mu$. More precisely, let $\mathbf{e}_1=(1,0,\dots,0)^\top\in\mathbb{R}^n$ and let $\mathbf{v}_k,1\le k\le n$ be the eigenvectors of $\cC$, then we have 
	% \begin{equation}\label{eq:spectral_measure}
		% \mu=\sum_{k=1}^n \delta_{e^{i\lambda_k}} \cdot \left|\langle \mathbf{e}_1,  \mathbf{v}_k\rangle\right|^2.
		% \end{equation}
	% Here $e^{i\lambda_k},1\le k\le n$ represents both the support of $\mu$ and the eigenvalues of $\cC$. Note that the above definition of spectral measure works for all unitary matrices $M$ for which $\mathbf{e}_1$ is cyclic. 
	In the following proposition, we summarize the connection between CMV matrices, probability measures on $\partial \D$, and the orthogonal polynomials associated with these probability measures.

	\begin{proposition}[\cite{KK,OPUC1}]\label{prop:CMV1}
		
		Suppose that ${\bf U}$ is an $n\times n$ unitary matrix for which $\mathbf{e}_1$ is cyclic. Let $\mu$ be its spectral measure with respect to $\mathbf{e}_1$, and denote its Verblunsky coefficients by $\alpha_k,\,  0\le k\le n-1$. Then there is a unitary matrix ${\bf V}$ with $V\mathbf{e}_1=\mathbf{e}_1$ so that ${\bf U}={\bf V} \cC(\alpha_0,\dots, \alpha_{n-1}) {\bf V}^{-1}$. 
		% In particular, $U$ and $\cC$ are unitarily equivalent, and they have the same spectral measure with respect to $\mathbf{e}_1$.
		Moreover, let $\Phi_k,0\le k\le n$ be the monic orthogonal polynomials of $\mu$, then we have
		\begin{equation}\label{eq:CMV_characteristic}
			\det(zI_k-\mathcal{C}(\alpha_0,\dots,\alpha_{k-1})) = \Phi_k(z),\quad 1\le k \le n.
		\end{equation}
	\end{proposition}
	
	To describe the matrix models for rank-one truncations and multiplicative perturbations of unitary ensembles, we introduce a simple transformation on the sequence of Verblunsky coefficients and the corresponding CMV matrix below. 
	
	\begin{definition}\label{def:rev_Verblunsky}
		Suppose $\mu$ is a finitely supported probability measure on $\partial\D$. Let $(\alpha_0, \dots, \alpha_{n-1})\in \D^{n-1} \times \partial \D$ and $\cC$ be the corresponding  Verblunsky coefficients and the CMV matrix with respect to $\mu$. We define the reversed version of the Verblunsky coefficients as 
		\begin{align}\label{eq:rev_Verblunsky}
			(\rev \alpha_0,\rev  \alpha_1,\dots, \rev \alpha_{n-1}):= (-\alpha_{n-1} \bar \alpha_{n-2}, -\alpha_{n-1} \bar \alpha_{n-3}, \dots, -\alpha_{n-1}\bar \alpha_0, \alpha_{n-1}). 
		\end{align}
		Then we denote by $\rev  \mu$ and $\rev  \cC$ the probability measure and CMV matrix corresponding to the sequence  $\rev \alpha_k,0\le k\le n-1$. We call $\rev  \mu$ and $\rev \cC$ the reversed version of $\mu$ and $\cC$, respectively. 
	\end{definition}
	
	In Killip-Nenciu \cite{KillipNenciu}, it was shown in Proposition B.2 that $\cC$ and $\rev  \cC$ have the same eigenvalues. In \cite{KK} Killip and Kozhan observed that the truncated and multiplicative perturbation of a CMV matrix is unitarily equivalent to a sub-unitary CMV matrix model. 
	
	\begin{proposition}[\cite{KK}]\label{prop:CMV3}
		Consider the same setup as in Proposition \ref{prop:CMV1} so that the unitary matrices ${\bf U}$ and $\cC(\alpha_0,\dots,\alpha_{n-1})$ are unitarily equivalent, and $n\ge 2$. Then there exists a unitary matrix ${\bf V}$ such that
		\begin{align}\label{eq:CMV_trunc}
			\trunc{{\bf U}} = \begin{cases}
				{\bf V}\cC(\rev \alpha_0, \dots, \rev \alpha_{n-2}){\bf V}^{-1} &\mbox{if $n$ is even,}\\
				{\bf V}\cC(\rev \alpha_0, \dots, \rev \alpha_{n-2})^\top {\bf V}^{-1}&\mbox{if $n$ is odd.}
			\end{cases}
		\end{align}
		For $r\in [-1,1]$, there exists a unitary matrix ${\bf V}$ such that 
		\begin{align}\label{eq:CMV_pert}
			{\bf U}^{[r]}=\begin{cases}
				{\bf V}\cC(\rev \alpha_0, \dots, \rev \alpha_{n-2},r\rev \alpha_{n-1}){\bf V}^{-1} &\mbox{if $n$ is even,}\\
				{\bf V}\cC(\rev \alpha_0, \dots, \rev \alpha_{n-2},r\rev \alpha_{n-1})^\top {\bf V}^{-1}&\mbox{if $n$ is odd.}
			\end{cases}
		\end{align}
	\end{proposition}

	\subsection{Matrix models for the full and truncated circular beta ensembles}
	In this subsection, we explicitly describe the matrix models for the circular beta ensembles constructed by Killip and Nenciu  \cite{KillipNenciu}, as well as their rank-one truncated and multiplicatively perturbed version introduced by Killip and Kozhan \cite{KK}.

	Recall the definition of the size $n$ circular beta ensemble with density function \eqref{eq:cbe}. We define the size $n$ Killip-Nenciu measure as a random probability measure on $\partial \D$ as
	\begin{equation}\label{eq:KillipNenciu}
		\mu^{\textup{KN}}_{n,\beta} =\sum_{j=1}^n \pi_j\delta_{e^{i\theta_j}},
	\end{equation}
	where the support is distributed as the size $n$ circular beta ensemble, and the weights $\pi_j, 1\le j\le n$ are
	chosen according to Dirichlet$(\beta/2,\dots,\beta/2)$ distribution, independently of the support.

	The joint distribution of the Verblunsky coefficients with respect to $\mu_{n,\beta}^{\textup{KN}}$ was computed in \cite{KillipNenciu}.
	
	%  \begin{definition}\label{def:Theta}
		% 	For  $a>0$ we denote by $\Theta(a+1)$ the distribution on $\D$ that has probability density function
		% 	\begin{align*}
			% 	\tfrac{a}{2\pi}(1-|z|^2)^{a/2-1}.
			% 	\end{align*}
		% 	We extend the definition for $a=0$ as follows: $\Theta(1)$ is the uniform distribution on $\partial \D$. 
		% \end{definition}
	
	\begin{proposition}[\cite{KillipNenciu}]\label{prop:circ_Verblunsky}
		For fixed $n$ and $\beta>0$, the sequence of Verblunsky coefficients $\alpha_k,0\le k\le n-1$ of $\mu_{n,\beta}^{\textup{KN}}$ are independent, rotationally invariant such that $|\alpha_k|^2\sim \textup{Beta}(1,\frac{\beta}{2}(n-k-1))$ for $0\le k\le n-2$ and $\alpha_{n-1}\sim \textup{Unif}(\partial\D)$. 
	\end{proposition}
	Recall the reversed version of a sequence of Verblunsky coefficients \eqref{eq:rev_Verblunsky}. By the independence and rotational invariance of $\alpha_k,0\le k\le n-1$, we have $\wtl \alpha_{k},0\le k\le n-1$ are independent, rotationally invariant such that $|\wtl\alpha_k|^2\sim \textup{Beta}(1,\frac{\beta}{2}(k+1))$ for $0\le k\le n-2$ and $\wtl\alpha_{n-1}\sim\textup{Unif}(\partial\D)$. Note that the spectral weights of the $\wtl\alpha_k,0\le k\le n-1$ are no longer Dirichlet-distributed. The advantage of the reversed ordering is that now the Verblunsky coefficients can be coupled in the same probability space. See Assumption \ref{assumption} below for more discussions.
	Together with Proposition \ref{prop:CMV3}, this yields the following result of Killip and Kozhan \cite{KK}.

	\begin{proposition}[\cite{KK}]\label{prop:tcirc_matrix}
		For fixed $n\ge 1$ and $\beta>0$, let $\alpha_k,0\le k\le n-1$ be independent, rotationally invariant random variables such that $|\alpha_k|^2\sim \textup{Beta}(1,\frac{\beta}{2}(k+1))$. Then the joint eigenvalue distribution of the sub-unitary CMV matrix  $\cC(\alpha_0,\alpha_1,\dots,\alpha_{n-1})$ 
		is the same as that of $\trunc{\Circ_{n+1,\beta}}$, which is given by the size $n$ truncated circular beta ensemble \eqref{eq:tcbe}.
		
		Let $\eta\sim \textup{Unif}(\partial\D)$ be independent of $\alpha_k,k\ge 0$. For fixed $r\in[-1,1]$, the sub-unitary CMV matrix $\cC(\alpha_0,\alpha_1,\dots,\alpha_{n-2},r\eta )$ has the same joint eigenvalue distribution as the multiplicatively perturbed model matrix $\Circ_{n,\beta}^{[r]}$. 
	\end{proposition}

	\section{Limits of the Blaschke products}\label{sec:proof_main}
	
	This section establishes the scaling limit of the Blaschke product associated with the truncated and multiplicatively perturbed circular beta ensembles. We first show the weak convergence of the Taylor coefficients of the Blaschke products in Section \ref{subsec:weak_conv_coeff}, and then upgrade this to uniform convergence on compacts in law by proving the necessary tightness result in Section \ref{subsec:tight}. These results yield proofs of Theorem \ref{thm:main} and Corollary \ref{cor:bulk}. Using a similar argument, we complete the proof of Theorem \ref{thm:perturb} in Section \ref{subsec:perturb}.

	\subsection{Weak convergence of the Taylor coefficients}\label{subsec:weak_conv_coeff}
	
	Recall the Blaschke product defined in Theorem \ref{thm:main}. By Propositions \ref{prop:CMV1} and \ref{prop:CMV3}, we have 
	\[
	B_n(z) = \frac{\Phi_n(z)}{\Phi_n^*(z)}.
	\]
	Since the zeros of $\Phi_n^*(z)$ lie outside of the unit disk, the Blaschke product $B_n(z)$ is well defined for $z\in\D$ and $\{z\in\D:B_n(z)=0\}=\{z\in\D:\Phi_n(z)=0\}$. Expanding $B_n(z)$ in a Taylor series around $z=0$, we write  
	\[
	B_n(z) = \sum_{k\ge 0} b_k^{(n)}z^k,\qquad z\in\D.
	\]
	The goal of this section is to prove the following statement. 
	\begin{proposition}\label{prop:weak_conv_coeff}
		Under the setup as in Theorem \ref{thm:main}, let $b_k^{(n)}$ be the $k$-th Taylor coefficient of $B_n(z)$. Then for any fixed $\ell \ge 0$, we have
		\begin{align*}
			\sqrt{n}(b_0^{(n)},b_1^{(n)},\dots, b_\ell^{(n)}) \Rightarrow \sqrt{\frac{2}{\beta}}(\xi_0,\xi_1,\dots,\xi_\ell)\quad \text{in law as $n\to\infty$},
		\end{align*}
		where $\xi_i, 0\le i\le \ell$ are i.i.d.~standard complex Gaussian variables. 
	\end{proposition}
	
	The proof of Proposition~\ref{prop:weak_conv_coeff} relies on two main ingredients: a martingale structure of the logarithm of the reversed orthogonal polynomials $\Phi_n^*$, and a Gaussian approximation for the Verblunsky coefficients. Throughout the paper, we work with the following probability space.
	\begin{assumption}\label{assumption}
		Let $\{\alpha_j,j\ge 0\}$  be a sequence of independent, rotationally invariant random variables such that 
		\[
		|\alpha_j|^2\sim \textup{Beta}(1,\beta_j),\qquad \beta_j=\tfrac{\beta}{2}(j+1)\qquad \text{for $j\ge 0$,}
		\]
		and set $\alpha_{-1}\equiv -1$.
		Let $\eta \sim \textup{Unif}\,(\partial \D)$ be independent of $\{\alpha_j,j\ge 0\}$.
		In this probability space, we introduce the filtration
		\begin{align*}
			\mathcal{F}_n = \sigma(\alpha_0,\cdots,\alpha_{n-1}),\qquad n\ge 0,
		\end{align*}
		and define the CMV matrices
		\begin{align*}
			\cC_n = \cC(\alpha_0,\alpha_1,\cdots,\alpha_{n-1}),\qquad \cC_n^{[r]} = \cC(\alpha_0,\alpha_1,\cdots,\alpha_{n-2},r\eta).
		\end{align*}
	\end{assumption}
	By Proposition \ref{prop:tcirc_matrix}, the eigenvalues of $\cC_n$ and $\cC_n^{[r]},r\in[-1,1]$ are distributed as the size $n$ truncated and multiplicatively perturbed circular beta ensemble, respectively. As already observed in \cite{CMN,CN_GMC}, the probability space defined in Assumption \ref{assumption} provides a natural coupling of full and truncated circular beta ensembles of all dimensions.
	
	Define $\log \Phi_k^*$ as the unique holomorphic logarithm on $\D$ such that $\log\Phi_k^*(0)=0$.
	By the Szeg\H o recursion \eqref{eq:Szego1}, for $k\ge 0, z\in\D$ we have 
	\begin{align*}\label{eq:log*_recursion}
		\log \Phi_{k+1}^*(z) = \log \Phi_k^*(z) + \log \left(1-\alpha_k z \frac{\Phi_k(z)}{\Phi_k^*(z)}\right),
	\end{align*}
	where we take the principal branch of the logarithm in the second term of the equality.
	By definition, $\Phi_k$ and $\Phi_k^*$ are $\mathcal{F}_k$-measurable, and $\alpha_k$ is independent of $\mathcal{F}_k$. Using the rotational invariance of $\alpha_k$, the integrability of $\log(1-|\alpha_k|)$ and the fact that $|zB_k(z)|<1$ for $z\in\D$, one can check that $\{\log \Phi_k^*(z), z\in\D\}$ is a $\mathcal{F}_k$-martingale. In \cite{JM15}, Jiang and Matsumoto computed the moments of traces of circular beta ensembles and established Central Limit Theorems for linear statistics of polynomials. (See also \cite{Li_CLT} for a direct proof of the CLTs by the first author.) By controlling the error between $\log \Phi_n^*$ and the logarithm of the reversed characteristic polynomial of the full C$\beta$E, and by establishing suitable integrability properties of $\log \Phi_k^*$, Chhaibi and Najnudel \cite{CN_GMC} characterized the limit $\log \Phi_\infty^*$ as follows. 
	
	\begin{proposition}[Propositions 3.1 and 3.6 of \cite{CN_GMC}]\label{prop:martingale_conv}
		Let $\Phi_k^*, k\ge 0$ be defined as above. Then we have that $\log \Phi_k^*$ converges to $\log \Phi_\infty^*$ a.s.~as $k\to\infty$, uniformly on compacts in $\D$. Moreover, the limit has the same distribution as 
		\begin{equation*}
			\log \Phi_\infty^*(z) \ed \sqrt{\frac{2}{\beta}}G^{\mathbb{C}}(z):=\sqrt{\frac{2}{\beta}}\sum_{k=1}^{\infty} \frac{z^k}{\sqrt{k}}\xi_k,
		\end{equation*}
		where $\xi_k$ are i.i.d.~standard complex Gaussian variables.
	\end{proposition}
	Note that Proposition 3.1 of \cite{CN_GMC} is stated for $\beta > 2$ due to its connection with Gaussian Multiplicative Chaos on $\partial\D$. Applying the arguments used in the proof of Proposition 3.1 of \cite{CN_GMC} with Proposition 3.6 of \cite{CN_GMC} and the CLTs for circular beta ensembles established in \cite{JM15} yields the general convergence result for $\beta>0$.
	
	We now turn to the Gaussian approximation of the Verblunsky coefficients. The next result follows directly from the rotational invariance and a CLT type approximation for a Beta distributed random variable.
	\begin{fact}\label{fact:Gaussianity}
		Let $\alpha$ be a rotationally invariant random variable such that $|\alpha|^2\sim\textup{Beta}(1,a)$, then
		\begin{align*}
			\sqrt{a}\alpha\Rightarrow \xi
		\end{align*}
		in law as $a\to\infty$, where $\xi$ is a standard complex Gaussian variable.
	\end{fact}
	Indeed, using the fact that $|\xi|^2$ has a standard exponential distribution, one can construct an explicit coupling between the Verblunsky coefficients defined in Assumption \ref{assumption} and an i.i.d.~sequence of standard complex Gaussian variables, see Remark \ref{rmk:explicit_coupling} below. However, this coupling does not substantially improve the convergence of the Blaschke product. For our purpose, the weak convergence stated in Fact \ref{fact:Gaussianity} is enough.
	
	We are now ready to prove Proposition \ref{prop:weak_conv_coeff}.
	
	\begin{proof}[Proof of Proposition \ref{prop:weak_conv_coeff}]
		By applying Szeg\H o recursion \eqref{eq:Szego1} repeatedly, we have 
		\begin{align*}
			\Phi_{n}(z) &= z\Phi_{n-1}(z) - \bar \alpha_{n-1} \Phi_{n-1}^*(z)\\
			&=z^2\Phi_{n-2}(z) - z\bar\alpha_{n-2}\Phi_{n-2}^*(z) -\bar \alpha_{n-1}\Phi_{n-1}^*(z)\\
			&=...\\
			&=z^n - z^{n-1}\bar\alpha_0\Phi_0^*(z) - z^{n-2}\bar \alpha_1\Phi_1^*(z) -\dots - z\bar \alpha_{n-2}\Phi_{n-2}^*(z) - \bar\alpha_{n-1}\Phi_{n-1}^*(z)
			\\ & =z^n  -\sum_{j=0}^{n-1} \bar \alpha_{n-1-j}\Phi_{n-1-j}^*(z)  z^{j}
		\end{align*}
		where we used $\alpha_{-1}=-1$ by convention and set $\Phi_{-1}=\Phi_{-1}^*=1$ for the ease of notation. Therefore, we have
		\begin{equation}\label{eqn::expansion}
			\sqrt nB_n(z)
			% =z^{n}\sqrt n\frac{1}{\Phi_{n}^*} - z^{n-1}\sqrt n\bar\alpha_0\frac{\Phi_0^*}{\Phi_{n}^*} - z^{n-2}\sqrt n\bar\alpha_1\frac{\Phi_1^*}{\Phi_{n}^*} -\dots - z\sqrt n\bar\alpha_{n-2}\frac{\Phi_{n-2}^*}{\Phi_{n}^*} - \sqrt n\bar\alpha_{n-1}\frac{\Phi_{n-1}^*}{\Phi_{n}^*}.
			=-\sum_{j=0}^n \sqrt{n}\bar\alpha_{n-1-j}\frac{\Phi_{n-1-j}^*(z)}{\Phi_n^*(z)}z^j.
		\end{equation}
		
		Recall that $b_k^{(n)}$ denotes the Taylor coefficient of $B_n$ of degree $k$ at zero. Fix $\ell\ge 0$, for any $0<r<1$ and $k\le \ell<n$, by the residue theorem we have 
		\begin{align}\label{eq:Taylor_expan}
			\sqrt{n}b_k^{(n)}=\frac{1}{2\pi i}\int_{|z|=r}\sqrt nB_n(z)\frac{dz}{z^{k+1}}=-\frac{1}{2\pi i}\sum_{j=0}^{n}\sqrt n\bar\alpha_{n-1-j}\int_{|z|=r}\frac{\Phi_{n-1-j}^*(z)}{\Phi_{n}^*(z)}\frac{dz}{z^{k+1-j}}.
		\end{align}
		Note that for $j> k$ the integrand on the right hand side of \eqref{eq:Taylor_expan} is holomorphic, hence the corresponding integral is zero.
		Proposition \ref{prop:martingale_conv} implies that $\Phi^*_{n}$ converges to $\Phi_\infty^*$ uniformly on compacts in $\D$: moreover, 
		$\Phi_\infty^*$ does not vanish in $\D$. Using this convergence, the fact that the zeros of $\Phi_k^*$, $k \geq 0$ lie outside the unit disk, and Cauchy's integral formula we obtain that
		\begin{equation}\label{eqn::coe_aux}
			\lim_{n\to\infty}\frac{1}{2\pi i}\int_{|z|=r}\frac{\Phi_{n-1-j}^*(z)}{\Phi_{n}^*(z)}\frac{dz}{z^{k+1-j}}=\mathbbm{1}_{\{j=k\}}.
		\end{equation}
		Recall the distribution of $\alpha_k$ as in Assumption \ref{assumption}, by Fact \ref{fact:Gaussianity} we have
		\begin{equation}\label{eqn::coe_Gaussian}
			(\sqrt n\bar \alpha_{n-1},\ldots,\sqrt n\bar \alpha_{n-\ell-1}) \Rightarrow \sqrt{\frac{2}{\beta}}(\xi_0,\xi_1,\dots,\xi_\ell)
		\end{equation}
		jointly in distribution as $n\to\infty$, where $\xi_i,0\le i\le \ell$ are i.i.d.~standard complex Gaussian variables. Applying Slutsky's theorem with \eqref{eqn::coe_aux} and \eqref{eqn::coe_Gaussian} completes the proof of the statement.
	\end{proof}
	
	We conclude this section with a brief discussion of an explicit coupling between the Verblunsky coefficients and the i.i.d.~complex Gaussian variables.
	\begin{remark}\label{rmk:explicit_coupling}
		Let $\xi_k, k\ge0$ be a sequence of i.i.d.~standard complex Gaussian variables. Define  
		\[
		\alpha_k:= \frac{\xi_k}{|\xi_k|}\sqrt{1-\exp(-\beta_k^{-1}|\xi_k|^2)},\qquad \beta_k:=\frac{\beta}{2}(k+1),\quad k\ge0,
		\]
		then one can check that $\alpha_k,k\ge 0$ is a sequence of independent, rotationally invariant random variables such that $|\alpha_k|^2\sim\textup{Beta}(1,\beta_k)$. For any $n\ge 1$, let $B_n(z)$ be the Blaschke product of the truncated C$\beta$E defined via $\alpha_0,\cdots,\alpha_{n-1}$ under this coupling. Define
		\begin{align*}
			G_n (z) =  -\sqrt{\tfrac{2}{\beta}} \left(\sum_{j=0}^{n-1}\ol{\xi}_{n-1-j}z^j + \sum_{j=n}^\infty \ol{\xi}_jz^j\right), 
		\end{align*}
		then $G_n\ed \sqrt{\frac{2}{\beta}}f_{\textup{GAF}}(z)$. By the triangle inequality, one can bound the error $|\sqrt{n}B_n-G_n|$ by controlling the following terms
		\begin{align*}
			\Big|\sqrt{n}B_n+\sqrt{n} \sum_{j=0}^{n-1}\ol{\alpha}_{n-1-j}z^j\Big|,\qquad \Big|\sum_{j=0}^{n-1}(\sqrt{n}\ol{\alpha}_{n-1-j}-\sqrt{\tfrac{2}{\beta}}\ol{\xi}_{n-1-j})z^j\Big|,\qquad \Big|\sum_{j\ge n}\bar \xi_j z^j\Big|.
		\end{align*}
		These terms can be estimated using the bounds on $|\frac{\Phi_{n-1-j}^*}{\Phi_n^*}-1|$ and $|\sqrt{n}\alpha_{n-j-1}-\sqrt{\tfrac{2}{\beta}}\xi_{n-1-j}|$, and the fact that $M=\sup_{k\ge 0} \frac{|\xi_k|}{\sqrt{\log(2+k)}}$ is a.s.~finite. Combining these estimates yields an upper bound on $\|\sqrt{n}B_n(z)-G_n(z)\|_r$ with an explicit convergence rate. This would lead to $|\sqrt{n}B_n(z)-G_n(z)|\to 0$ a.s.~uniformly on compacts in $\D$ as $n\to\infty$. However, since the sequence of GAFs $G_n(z)$ varies with $n$, this estimate alone would not establish convergence in probability of $\sqrt{n}B_n$ to a fixed GAF. Moreover, the argument relies on the explicit coupling between the Verblunsky coefficients and the i.i.d.~standard complex Gaussian variables and could not easily be extended to a more general framework.
	\end{remark}

	\subsection{Tightness results}\label{subsec:tight}
	
	Proposition \ref{prop:weak_conv_coeff} establishes the weak convergence of the Taylor coefficients in the power series. In general, this result alone does not imply the convergence of zeros. For instance, the sequence $f_n(z)= 2^{n^2}z^n$ only converges at zero  (see e.g.~\cite{Krishnapur,HKPV} for more discussions). To complete the proof of Theorem \ref{thm:main}, we need to show weak convergence of $\sqrt{n}B_n(\cdot)$ with respect to the topology of uniform convergence on compacts in $\D$.

	We first introduce some notation. Denote by $\mathcal{H}=\cH(\D)$ the set of holomorphic functions on the unit disk. Define $d_\mathcal{H}$ to be the following metric on $\mathcal{H}$: for every $f,g\in\mathcal{H}$, we define
	\begin{equation}\label{eq:metric}
		d_\mathcal{H}(f,g)=\sum_{m\ge 1} 2^{-m}\min\left\{1,\|f-g\|_{r_m}\right\},\qquad \|f\|_r:=\sup_{z:|z|\le r}|f(z)|,\qquad r_m= 1-2^{-m}.
	\end{equation}
	Note that $(\mathcal{H},d_\mathcal{H})$ is a complete and separable metric space. The following statement gives the required tightness.
	
	\begin{proposition}\label{prop:tight}
		Consider the same setup as in Theorem \ref{thm:main}, the Blaschke product $\sqrt{n}B_n(\cdot), n\ge 1$ is tight as a sequence of random analytic functions with respect to the metric $d_{\mathcal{H}}$.
	\end{proposition}
	\begin{proof}
		Recall the Szeg\H o recursion \eqref{eq:Szego1}, we have 
		\begin{align*}
			B_{n+1}(z) = \frac{zB_n(z)-\bar \alpha_n}{1-\alpha_n zB_n(z)}.
		\end{align*}
		Using the fact that 
		\begin{align*}
			\left|\frac{w-\bar z}{1-zw}\right|\le \frac{|w|+|z|}{1+|wz|}\le |w|+|z|,\qquad w,z\in\D,
		\end{align*}
		we get
		\begin{align}\label{eq:B_rec_bound}
			|B_{n+1}(z)|\le |zB_n(z)|+|\alpha_{n}|.
		\end{align}
		For $0<r<1$ recall the norm $\|\cdot\|_r$ defined in \eqref{eq:metric}. By iterating the bound \eqref{eq:B_rec_bound} we have 
		\begin{align*}
			\|B_{n}\|_r \le r\|B_{n-1}\|_r +|\alpha_{n-1}|\le \cdots\le r^n+ \sum_{j=0}^{n-1} r^j|\alpha_{n-1-j}|. 
		\end{align*}
		Recall that $|\alpha_k|^2\sim \textup{Beta}(1,\beta_k)$, we have $\ev[|\alpha_k|^2]=\frac{1}{1+\beta_k}$  with $\beta_k=\frac{\beta}{2}(k+1)$. By Minkowski's inequality we have  
		\begin{align*}
			\ev[\|\sqrt{n}B_n\|_r^2]^{1/2}\le c_\beta \sqrt{n}\sum_{j=0}^n\frac{r^j}{\sqrt{n+1-j}},
		\end{align*}
		where $c_\beta$ is a constant depending only on $\beta$. By splitting the summation at $j=\lfloor n/2\rfloor$, we can find a deterministic constant $c>0$ such that 
		\[
		\sqrt{n}\sum_{j=0}^{\lfloor n/2\rfloor} \frac{r^j}{\sqrt{n+1-j}}\le \frac{c}{1-r},\qquad \sqrt{n}\sum_{j=\lfloor n/2\rfloor+1}^n \frac{r^j}{\sqrt{n+1-j}}\le cnr^{n/2}.
		\]
		Since $nr^{n/2}\le \frac{1}{1-r}$, we may enlarge $c_\beta$ (if necessary) to conclude that 
		\[
		\ev[\|\sqrt{n}B_n\|_r^2]\le \frac{c_\beta}{(1-r)^2},
		\]
		uniformly in $n$. We claim that this uniform bound also implies the tightness of the sequence $\sqrt{n}B_n,n\ge 1$. Indeed, for any $\eps>0$, we set $M_{m,\eps}=\sqrt{c_\beta 2^{3m}/\eps}$ and consider the compact set 
		$\cH_{\eps}:=\cap_{m\ge 1}\{f\in\cH:\|f\|_{r_m}\le M_{m,\eps}\}$. By Markov's inequality, we have 
		\begin{align*}
			\pr(\|\sqrt{n}B_n\|_{r_m}\ge M_{m,\eps})\le \eps 2^{-m}.
		\end{align*}
		Taking the union bound in $m$ gives that $\pr(\sqrt{n}B_n\notin \cH_{\eps})\le \eps$ for all $n$. This completes the proof of the statement. 
	\end{proof}

	We are now ready to prove Theorem \ref{thm:main}.
	\begin{proof}[Proof of Theorem~\ref{thm:main}]
		Proposition \ref{prop:weak_conv_coeff} proves the weak convergence of the Taylor coefficients in the power series, and identifies any subsequential limit. Proposition \ref{prop:tight} gives the necessary tightness to conclude the convergence in distribution in $\cH$, with respect to the topology of uniform convergence on compact sets of $\D$. Combining these two results completes the proof.
	\end{proof}
	To prove Corollary \ref{cor:bulk}, we invoke the following fact from complex analysis. 
	\begin{fact}\label{fact:Hurwitz}
		For $f\in\cH\setminus\{0\}$, let $\mathcal{Z}(f)=\sum_{w\in\D:f(w)=0}\textup{mult}_f(w)\delta_w$, where $\textup{mult}_f(w)$ counts the multiplicity of the zero $w$. Then the map $\mathcal{Z}$ is continuous from $\cH\setminus\{0\}$ with the topology of uniform convergence on compact sets of $\D$, to the space of locally finite point measures on $\D$ with respect to the vague topology.
		
		Consequently, if $f_n,n\ge 1$ and $f$ are random holomorphic functions that are a.s.~not identically equal to 0, and  $f_n\Rightarrow f$ in law as $n\to\infty$ with respect to the topology of uniform convergence on compact sets of $\D$, then $\mathcal{Z}(f_n)\Rightarrow \mathcal{Z}(f)$ in the vague topology. 
	\end{fact}
	\begin{proof}
		The deterministic statement follows from Hurwitz’s theorem (see e.g.~Theorem VII.2.5 in \cite{Conway1978}). The probabilistic statement then follows from the continuous mapping theorem.
	\end{proof}
	\begin{proof}[Proof of Corollary \ref{cor:bulk}]
		Since $B_n(z),n\ge 1$ and $f_{\textup{GAF}}(z)$ are a.s.~not identically zero, the statement follows from Theorem \ref{thm:main} and Fact \ref{fact:Hurwitz}.
	\end{proof}
	\subsection{Limits of the multiplicatively perturbed circular beta ensemble}\label{subsec:perturb}
	Recall the rank-one multiplicative perturbation of a unitary matrix ${\bf M}$ defined in \eqref{def:mult_pert}.
	In \cite{ForIpsen}, Forrester and Ipsen studied the asymptotic behavior of the zeros of ${\bf M}_n^{[r_n]}$, where ${\bf M}_n$ is a size $n$ CUE (corresponding to the $\beta=2$ case) and $r_n=O(n^{-1/2})$. Their result can be summarized as follows.
	\begin{proposition}[\cite{ForIpsen}]\label{prop:ForIpsen}
		Let ${\bf M}_n\in U(n)$ be uniformly chosen with Haar measure on $U(n)$, and set $r_n = r/\sqrt{n}$ with $r$ fixed. Then the eigenvalues of ${\bf M}_n^{[r_n]}$ converge weakly as $n\to\infty$ to the zeros of the random series 
		\begin{align}\label{eq:GAF^mu}
			f_{\textup{GAF},r}(z):= r +\sum_{j=1}^\infty \xi_j z^j,
		\end{align}
		Here $\xi_j,j\ge 1$ is a sequence of i.i.d.~standard complex Gaussian variables.  
	\end{proposition}
	The function $f_{\textup{GAF},r}(\cdot)$ can be viewed as the Gaussian analytic function $f_{\textup{GAF}}(\cdot)$  conditioned on its value at $0$. 
	Following the result of Forrester and Ipsen, it is natural to study the asymptotic behavior of zeros of the rank-one multiplicative perturbation of the circular beta ensemble for general $\beta>0$. To this end, we derive the scaling limit of the normalized Blaschke product of the perturbed model using an argument similar to the proof of Theorem \ref{thm:main}. This convergence is strong enough to extend Proposition \ref{prop:ForIpsen} to general $\beta>0$. 
	
	\begin{proof}[Proof of Theorem \ref{thm:perturb}]
		Fix $r\in\R$ and we assume $n$ is large enough such that $|r_n|=|\sqrt{\frac{2}{\beta n}}r|<1$. Let $\alpha_k,k\ge 0$ and $\eta$ be defined as in Assumption \ref{assumption}. By Propositions \ref{prop:CMV3} and \ref{prop:tcirc_matrix}, the eigenvalues of $\Circ_{n,\beta}^{[r_n]}$ have the same joint distribution as the (sub-unitary) CMV matrix $\cC(\alpha_0,\dots,\alpha_{n-2}, r_n\eta )$. 
		
		Let $\Phi_{k}^{[r_n]},0\le k\le n$ be the monic orthogonal polynomials with respect to the sequence of Verblunsky coefficients $(\alpha_0,\alpha_1,\dots, \alpha_{n-2},r_n\eta )$, and let $(\Phi_{k}^{[r_n]})^*$ be the associated reversed polynomials. By the Szeg\H o recursion \eqref{eq:Szego1}, these polynomials do not depend on $r$ for $0\le k\le n-1$, and we may therefore suppress the $r$-dependence in the notation for these functions.
		Consequently, Theorem \ref{thm:main} and Proposition \ref{prop:martingale_conv} remain valid for the sequences $B_{n-1}$ and $(\Phi_{n-1}^{[r_n]})^*=\Phi_{n-1}^*$. 
		
		Applying Szeg\H o recursion \eqref{eq:Szego1} to $\Phi_n^{[r_n]}$ and $(\Phi_n^{[r_n]})^*$, we have 
		\begin{align*}
			\Phi_{n}^{[r_n]}(z) = z\Phi_{n-1}(z)- r_n\bar \eta \Phi_{n-1}^*(z),\qquad (\Phi_{n}^{[r_n]})^*(z) &= \Phi_{n-1}^*(z) -z r_n \eta \Phi_{n-1}(z),
		\end{align*}
		hence
		\begin{align*}
			\sqrt{n}B_{n}^{[r_n]}(z)=\sqrt{n}\frac{\Phi_{n}^{[r_n]}(z)}{(\Phi_{n}^{[r_n]})^*(z)} =\frac{z\sqrt{n}B_{n-1}(z)-\sqrt{\tfrac{2}{\beta}}r\bar \eta}{1-r_n\eta zB_{n-1}(z)}.
			% &= \sqrt{n}\frac{\Phi_{n}^{[r_n]}(z)}{\Phi_{n-1}^*(z)}\cdot\frac{\Phi_{n-1}^*(z)}{(\Phi_{n}^{[r_n]})^*(z)}=\sqrt{n}\frac{z\Phi_{n-1}(z)-r_n\bar\eta  \Phi_{n-1}^*(z)}{\Phi_{n-1}^*(z)}\cdot\frac{\Phi_{n-1}^*(z)}{(\Phi_{n}^{[r_n]})^*(z)}.
		\end{align*}
		Since $r_n=O(n^{-1/2})$ and $|B_{n-1}(z)|<1$ for $z\in\D$, the denominator converges to $1$ uniformly on compact sets of $\D$ (see also Theorem 1.7.4 of \cite{OPUC1}). By Theorem \ref{thm:main},  Slutsky's theorem and the independence of $\eta$ from $B_{n-1}$, we conclude that 
		\begin{align*}
			\sqrt{n}B_{n}^{[r_n]}(z) \Rightarrow \sqrt{\frac{2}{\beta}}(zf_{\textup{GAF}}(z) - r\bar \eta)
		\end{align*}
		% The recursion \eqref{eq:Szego_Phin*} together with Theorem \ref{thm:main} imply that $|\frac{(\Phi_{n}^{[r_n]})^*(z)}{\Phi_{n-1}^*(z)}-1|$ is of order $O(n^{-1})$ uniformly on compacts in $\D$. (One may also use Theorem 1.7.4 of \cite{OPUC1} directly to conclude that this difference is $o(n^{-1/2})$.)
		in distribution as $n\to\infty$ uniformly on compact sets of $\D$. Since $zf_{\textup{GAF}}(z) - r\bar \eta\ed r\eta+\sum_{j=1}^\infty \xi_j z^j$, we conclude the proof of Theorem \ref{thm:perturb}. 
		% \color{red}
		% We cannot get rid of $\eta$
		% at the level of functions, only at the level of their zeros. 
		% \color{black}\textcolor{blue}{Yun: Good point, you are right.}
		
		Repeating the proof of Corollary \ref{cor:bulk} and applying Theorem \ref{thm:perturb}, we conclude that the zeros of the multiplicatively perturbed C$\beta$E converge to those of the random series $zf_{\textup{GAF}}(z) - r\bar \eta$, independent of $\beta$. Using the fact that the Gaussian variables are rotation invariant, the limiting zero set has the same distribution as the zero set of the conditional Gaussian series $f_{\textup{GAF},r}(\cdot)$ as claimed.
	\end{proof}

	\section{Asymptotic independence of $\Phi_n^*$ and the bulk zeros }\label{sec:lss}
	Let $\Lambda_n=\{z_1^{(n)},z_{2}^{(n)},\dots,z_n^{(n)}\}$ be distributed according to the size $n$ truncated C$\beta$E and set $\Lambda = \{z\in\D:f_{\textup{GAF}}(z)=0\}$. Recall the reversed characteristic polynomial $\Phi_n^*$ of $\Lambda_n$. 
	The goal of this section is to prove the asymptotic independence of $\Phi_n^*$ and the linear statistics of $\Lambda_n$ of compactly supported test functions. Heuristically speaking, the reversed polynomial $\Phi_k^*$ is determined by the first $k$ Verblunsky coefficients, and is close to its limit $\Phi_\infty^*$ when $k$ is sufficiently large (see Proposition \ref{prop:martingale_conv}). In contrast, the bulk zeros, namely those lying inside a compact subset of $\D$, are mainly determined by the first few terms in the series expansion of $\sqrt{n}B_n$. As shown in Proposition \ref{prop:weak_conv_coeff}, these Taylor coefficients depend on the Verblunsky coefficients in reverse order.

	Before we present the result, let us introduce a sequence of auxiliary functions 
	\begin{equation}\label{eq:aux}
		B^{(m)}_m=0,\qquad B^{(m)}_{n+1}(z) = \frac{zB_n^{(m)}(z)-\bar \alpha_n}{1-\alpha_n z B_n^{(m)}(z)},\qquad n\ge m.
	\end{equation}
	The function $B^{(m)}_n$ depends on $\alpha_m,\alpha_{m+1},\cdots,\alpha_{n-1}$, and is independent of $\cF_m$. By the inequality $|w-\bar z|\le |1-w z|$, it is direct to check that $|B_n^{(m)}|\le 1$ on $\D$. Recall the metric $d_\cH$ defined in \eqref{eq:metric}, we obtain the following bound on $B_n-B_n^{(m)}$. 
	\begin{proposition}\label{prop:Bn_aux}
		Consider the probability space defined in Assumption \ref{assumption}. For $0<r<1$ and $m< n$, we have
		\begin{align*}
			\|B_n-B_n^{(m)}\|_r \le 2r^{n-m}.
		\end{align*}
		Moreover, for any fixed $m$, we have $d_\cH(\sqrt{n}B_n,\sqrt{n}B_n^{(m)})\to 0$ a.s.~as $n\to\infty$.
	\end{proposition}
	\begin{proof}
		By definition, we have 
		\begin{align*}
			B_{n+1}(z)-B_{n+1}^{(m)}(z) = \frac{\rho_n^2z(B_n(z)-B_n^{(m)}(z))}{(1-\alpha_n zB_n(z))(1-\alpha_n zB_n^{(m)}(z))}.
		\end{align*}
		Since $|B_n(z)|\le 1,|B_n^{(m)}(z)|\le 1$ for $z\in \D$, we get the difference $B_{n+1}-B_{n+1}^{(m)}$ is a multiple of $z$. By induction, one can check that $B_n-B_n^{(m)}$ is a multiple of $z^{n-m}$, i.e., it has a zero of order at least $n-m$ at $z=0$. This implies that the first $n-m$ Taylor coefficients of $B_n$ and $B_n^{(m)}$ agree with each other. Repeating the proof of the Schwarz lemma proves the first claim. More precisely, let $g(z) = z^{-(n-m)}(B_n(z)-B_n^{(m)}(z))$. For $|z|=r<1$, by the maximum principle we have $|g(z)| \le 2r^{-(n-m)}$. Letting $r\uparrow 1$ yields that $|B_n(z)-B_n^{(m)}(z)|\le 2|z|^{n-m}$ as claimed. 
		
		We now prove the second claim. For $n> m$ we have 
		\begin{align*}
			d_\cH(\sqrt{n}B_n,\sqrt{n}B_n^{(m)})&\le \sum_{j\ge 1}2^{-j}\min\{1,2\sqrt{n}r_j^{n-m}\}\le 2\int_0^1 \min\{1,2\sqrt{n}r^{n-m}\}dr.
		\end{align*}
		By splitting the integral at $r_0=(2\sqrt{n})^{-1/(n-m)}$, we get the upper bound
		\[
		d_\cH(\sqrt{n}B_n,\sqrt{n}B_n^{(m)})\le 2\Big(\frac{r_0}{n-m+1}+1-r_0\Big)\le \frac{c(1+\log n)}{n-m+1},
		\]
		% \[
		% d_\cH(\sqrt{n}B_n,\sqrt{n}B_n^{(m)})\le \frac{2\sqrt{n}r_0^{-1/(n-m+1)}}{n-m+1}+1-r_0\le c\frac{\log(2\sqrt{n})}{n-m},
		% \]
		where $c$ is an absolute constant independent of $n,m$.  Sending $n\to\infty$ completes the proof.
	\end{proof} 
	For a continuous function $g$ with compact support inside the unit disk, we denote by $L_g(\Lambda_n):=\sum_{z\in\Lambda_n}g(z)$ the linear statistics of $\Lambda_n$. The next result proves the asymptotic independence of $(\Phi_n^*,L_g(\Lambda_n))\in \cH\times \C$.
	\begin{proposition}\label{prop:asymp_ind}
		Consider the probability space in Assumption \ref{assumption}. For any continuous compactly supported test function $g$, we have $(\Phi_n^*,L_g(\Lambda_n))\Rightarrow (\Phi_\infty^*,L_g(\Lambda))$, where the limit $L_g(\Lambda)=\sum_{z\in\Lambda}g(z)$ is independent of $\Phi_\infty^*$.
	\end{proposition}
	\begin{proof}
		By Corollary \ref{cor:bulk} and Proposition \ref{prop:martingale_conv}, we have the marginal convergence $L_g(\Lambda_n)\Rightarrow L_g(\Lambda)$ and $\Phi_n^*\to\Phi_\infty^*$ a.s.~uniformly on compacts as $n\to\infty$. It remains to prove the joint convergence. 
		
		For any fixed $m$, recall the auxiliary function $B_n^{(m)}$ defined in \eqref{eq:aux}. For $n>m$ we set $\Lambda_n^{(m)}=\{z\in \D:B_n^{(m)}(z)=0\}$, and $L_g(\Lambda_n^{(m)})=\sum_{z\in \Lambda_n^{(m)}}g(z)$. By definition, we have $\Lambda_n^{(m)}$ and $L_g(\Lambda_n^{(m)})$ are independent of $\cF_m$.

		Let $F:\cH(\D)\to \C$ and $G:\C\mapsto \C$ be bounded continuous functions. By the triangle inequality we have  
		\begin{align}
			&\big|\ev[F(\Phi_n^*)G(L_g(\Lambda_n))]-\ev[F(\Phi_\infty^*)]\ev[G(L_g(\Lambda))]\big|\notag\\
			&\quad\le  \ev\Big[\big|(F(\Phi_n^*)-F(\Phi_m^*))G(L_g(\Lambda_n))\big|\Big] + \ev\Big[\big|F(\Phi_m^*)(G(L_g(\Lambda_n))-G(L_g(\Lambda_n^{(m)}))\big|\Big]\\
			&\qquad + \big|\ev[F(\Phi_m^*)]-\ev[F(\Phi_\infty^*)]\big|\ev[|G(L_g(\Lambda_n^{(m)}))|]\label{eq:ind_triangle}\\
			& \qquad +\ev[|F(\Phi_\infty^*)|]\big|\ev[G(L_g(\Lambda_n^{(m)}))]-\ev[G(L_g(\Lambda))]\big|,
		\end{align}
		where in \eqref{eq:ind_triangle} we used the fact that $\Phi_m^*$ is independent of $\Lambda_n^{(m)}$.
		By Theorem \ref{thm:main} and Proposition \ref{prop:Bn_aux}, we have 
		\[
		(\sqrt{n}B_n,\sqrt{n}B_n^{(m)})\Rightarrow (X,X)\in \cH\times \cH,\qquad X\ed \sqrt{\tfrac{2}{\beta}}f_{\textup{GAF}},
		\]
		in law as $n\to\infty$, with respect to the topology of uniform convergence on compact sets of $\D$. By Fact \ref{fact:Hurwitz} again, we have $(L_g(\Lambda_n),L_g(\Lambda_n^{(m)}))\Rightarrow (L_g(\Lambda),L_g(\Lambda))$, and 
		$\Lambda_n^{(m)}\Rightarrow \Lambda$ as $n\to\infty $ in the vague topology on locally finite point measures on $\D$. Hence, for any fixed $m$ we have 
		\begin{align*}
			\ev\Big[\big|G(L_g(\Lambda_n))-G(L_g(\Lambda_n^{(m)}))\big|\Big]\to 0,\qquad \ev[G(L_g(\Lambda_n^{(m)}))]\to \ev[G(L_g(\Lambda))],
		\end{align*}
		as $n\to\infty$.
		Therefore, we have
		\begin{align*}
			\limsup_{n\to\infty}\big|\ev[F(\Phi_n^*)G(L_g(\Lambda_n))]-\ev[F(\Phi_\infty^*)]\ev[G(L_g(\Lambda))]\big|\le 2\|G\|_\infty \ev[|F(\Phi_\infty^*)-F(\Phi_m^*)|].
		\end{align*}
		Sending $m\to\infty$ and using the a.s.~uniformly on compacts convergence of $\Phi_m^*\to\Phi_\infty^*$ in Proposition \ref{prop:martingale_conv} completes the proof. 
	\end{proof}

	Note that by applying the argument used in the proof of Proposition \ref{prop:asymp_ind} to functionals of $(\Phi_n^*,\sqrt{n}B_n)$, one obtains the joint convergence
	\begin{align}\label{eq:joint_ind}
		(\Phi_n^*,\sqrt{n}B_n)\Rightarrow (\Phi_\infty^*,\sqrt{\tfrac{2}{\beta}}f_{\textup{GAF}}),\qquad \text{in law as $n\to\infty$,}
	\end{align}
	with respect to the compact-open topology, where the limiting functions $\Phi_\infty^*$ and $f_{\textup{GAF}}$ are independent. Using the fact that multiplication is continuous with respect to the compact-open topology, one may further deduce that 
	\begin{align*}
		\sqrt{n}\Phi_n\Rightarrow \sqrt{\tfrac{2}{\beta}}\exp\left(\sqrt{\tfrac{2}{\beta}}\sum_{j\ge1}\frac{\zeta_j}{\sqrt{j}}z^j\right)\cdot\left(\sum_{j\ge 0}\xi_j z^j\right)
	\end{align*}
	in law as $n\to\infty$, with respect to the topology of uniform convergence on compact sets of $\D$. Here $\zeta_j,j\ge 1$ and $\xi_j,j\ge 0$ are two sequences of i.i.d.~standard complex Gaussian variables and are independent of each other.

	We conclude the paper with a discussion of the possible extensions of our results.

	\begin{remark}
		Our results can be generalized to a large class of sequences of Verblunsky coefficients. Following the approach used in Propositions \ref{prop:weak_conv_coeff} and \ref{prop:tight}, we require a scaling limit (in distribution) and suitable moment bounds for the Verblunsky coefficients, together with the uniform-on-compact convergence of $\log \Phi^*_n$ to a limit $\log \Phi^*_{\infty}$. 
		In particular, we believe that our methods can be adapted to study the bulk limit of the rank-one truncations of the circular Jacobi beta ensemble and real orthogonal beta ensemble.
		
		In \cite{LV_truncated}, Li and Valk\'o also raised the question of the scaling limit of multi-rank truncation of circular ensembles. We expect the bulk limit to be a beta-generalization of Krishnapur's result on the matrix-valued Gaussian analytic functions \cite{Krishnapur}. We leave these extensions for future work.
	\end{remark}

	\bibliography{rmt}

	{{  
			\bigskip
			\footnotesize

			\noindent Yun Li, \textsc{School of Mathematical Sciences, Nankai University, China.}\\
			\textit{Email:} \texttt{liyun1723@hotmail.com}
			
			\medskip
			
			\noindent Mingchang Liu, \textsc{Academy for Multidisciplinary Studies, Capital Normal University, China}
			\\
			\textit{Email:} \texttt{liumc\_prob@163.com}
			
			\medskip
			
			\noindent Joseph Najnudel, \textsc{School of Mathematics, University of Bristol, United Kingdom} 
			\\
			\textit{Email:} \texttt{joseph.najnudel@bristol.ac.uk}
	}}
	
\end{document}